\documentclass[12pt]{amsart}
\usepackage{graphicx} 

\usepackage{amsfonts} 
\usepackage{amsmath} 
\usepackage{bbm}  
\usepackage{amsthm}
\usepackage{color}
 \usepackage{amssymb}
\usepackage{comment}
\usepackage{url}

\definecolor{bmcolor}{rgb}{0.9, 0.3, 0}

\definecolor{aacolor}{rgb}{0, .4, .4}

\definecolor{pmcolor}{rgb}{0.09, .45, .27}

\newcommand{\R}{\mathbb{R}}

\newcommand{\N}{\mathbb{N}}
\newcommand{\CC}{\mathbb{C}}
\newcommand{\DD}{\mathbb{D}}
\newcommand{\HH}{\mathcal{H}}

\newcommand{\ang}[1]{\langle #1 \rangle}

\newtheorem{thm}{Theorem}[section]
\newtheorem{cor}[thm]{Corollary}
\newtheorem{rem}[thm]{Remark}
\newtheorem{lem}[thm]{Lemma}
\newtheorem{defn}[thm]{Definition}
\newtheorem{prop}[thm]{Proposition}

\title{
Demystifying Carleson Frames}

\author{ Ilya Krishtal and Brendan Miller }

\begin{document}

\address{Department of Mathematical Sciences, Northern Illinois University, DeKalb, IL 60115 \\
email: ikrishtal@niu.edu, bmiller14@niu.edu}

\begin{abstract}
We study spanning properties of Carleson systems and prove a recent conjecture on frame subsequences of Carleson frames. In particular, we show that if $\{T^k\varphi\}_{k=0}^\infty$ is a Carleson frame, then every subsequence of the form $\{T^{Nk+j_k}\varphi\}_{k=0}^\infty$ where $N\in\N$ and $0 \leq j_k < N$ is also a frame.
\end{abstract}

\maketitle

\section{Introduction}

Motivated by applications in dynamical sampling \cite{ACMT17, ADK13,  AMJMP23}, O.~Christensen et al.~have developed a theory of dynamical frames, which are frames of orbits of a vector in a Hilbert space under the action of a linear operator \cite{CH19, CH23, CHP20, CHPS24}. A special kind of such frames was introduced in \cite{ACMT17} and called \emph{Carleson frames} in \cite{CHPS24}. Carleson frames possess very high redundancy and other mysterious properties that have not been observed in other classes of frames.  In this note, we begin to shed some light on excessive redundancy of Carleson frames by expanding the class of Carleson systems which are known to remain frames or form complete sequences. In particular, we prove a conjecture from \cite{CHPS24} showing that if $\{T^k\varphi\}_{k=0}^\infty$ is a Carleson frame with a positive definite operator $T$, then every subsequence of the form $\{T^{Nk+j_k}\varphi\}_{k=0}^\infty$, where $N\in\N$ and $0 \leq j_k <  N$, is also a frame.

Let us make the statement above more precise. 

We will call a complex sequence $\{z_j\}_{j = 0}^\infty$ a \emph{Carleson sequence} if $z_j \in \DD = \{z\in\CC: |z|<1\}$ and 
\begin{equation}\label{CCond}
\inf_{k} \prod_{j\neq k} \bigg|\frac{z_k - z_j}{1 - z_j\overline{z_k}} \bigg| =: \delta > 0.    
\end{equation}

Let $\{e_j\}_{j = 0}^\infty$ be an orthonormal basis for a (separable) Hilbert space $\HH$, and assume that $T$ is an operator on $\HH$ satisfying $Te_j = z_je_j$ for some Carleson sequence $\{z_j\}_{j =0}^\infty$. Note that hereinafter we index all sequences over $\N_0 = \N\cup\{0\}$ for the sake of convenience. 

According to \cite[Theorem~3.16]{ACMT17}, the system $\{T^{k}\varphi\}_{k = 0}^\infty$ forms a Carleson frame if 
\begin{equation}\label{cyclic_vect}
\varphi = \sum_{j \in \N_0} m_j\sqrt{1-|z_j|^2}e_j,
\end{equation}
where $\{m_j\}_{j \in \N_0}$ is a complex sequence such that $ 0 < C_1 \leq |m_j| \leq C_2$. We will assume without loss of generality that $m_j = 1$ for all $j\in\N_0$ (due to a change-of-basis argument).
Given a Carleson frame or a more general \emph{Carleson system} of the form $\{T^\lambda\varphi\}_{\lambda\in\Lambda}$, $\Lambda\subseteq[0,\infty)$, with $\varphi$ given by \eqref{cyclic_vect}, we shall refer to the sequence $\{z_j\}_{j =0}^\infty$ as its \emph{Carleson spectrum}. 

It was proven in \cite[Theorem~2.1]{CHPS24} that when $\{|z_j|\}_{j = 0}^\infty$ is strictly increasing, the system $\{T^{Nk}\varphi\}_{k=0}^\infty$ also forms a frame for any constant $N \in \N$. Furthermore, the authors of \cite{CHPS24} conjectured that if the Carleson spectrum is real and positive, i.e.~$z_j \in (0,1)$ 
for each $j \in \N$, this result can be extended to encompass Carleson systems of the form $\{T^{Nk+j_k}\varphi\}_{k =0}^\infty$ where $j_k \in \{0,1,...,N-1\}$. As mentioned above, the main result of this paper establishes this conjecture. In fact, we achieve a slightly more general result in Theorem \ref{theorem_main} as we consider fractional powers of $T$ by allowing non-integer $j_k\in [0, N)$. We also establish completeness results for certain systems with complex Carleson spectrum (Theorem \ref{theorem_completeness}) and estimate lower bounds for continuous Carleson frames (Theorem \ref{theorem_estimates}).

Let us recall a few standard frame-theoretic notions and results that we shall utilize in our proofs (see, e.g.~\cite{Christensen16}). 

For a sequence $\{f_k\}_{k =0}^\infty$ of elements in $\HH$, \emph{synthesis operator} $\Phi: \ell^2(\N_0) \to \HH$ is given by 
\[
(c_k)_{k = 0}^\infty \mapsto \sum_{k = 0}^\infty c_k f_k.
\]
The adjoint $\Phi^*: \HH\to \ell^2(\N_0)$ of the synthesis operator is given by $\Phi^*f = (\langle f, f_k\rangle)_{k \in \N_0}$ and is called \emph{analysis operator}.

 The following is a classical characterization of frames.
\begin{prop}\label{proposition_frame_equivalent}
    Suppose that $\{f_k\}_{k = 0}^\infty$ is a sequence in $\HH$. Then the following are equivalent: 
    \begin{enumerate}
        \item the sequence $\{f_k\}_{k = 0}^\infty$ forms a frame;
        \item the synthesis operator $\Phi: \ell^2(\N_0) \to \HH$ is well defined, bounded, and surjective;
        \item the analysis operator $\Phi^*: \HH \to \ell^2(\N_0)$ is well defined, bounded and bounded below. 
    \end{enumerate}
\end{prop}

As in \cite{CHPS24}, one of our key techniques is frame perturbation. In particular, we will use the following slight modification of \cite[Exercise 22.4]{Christensen16}. 

\begin{lem}\label{lemma_perturb}
    Let $\{f_k\}_{k = 0}^\infty$ be a frame for $\HH$ with frame bounds $B \geq A > 0$. If $V \subset \HH$ is a closed subspace, $P_V$ is an orthogonal projection onto $V$, and $\{h_k\}_{k = 0}^\infty$ is any sequence such that 
    \[
     \sum_{k = 0}^\infty \|P_V(f_k - h_k) \|^2 < A
    \]
    holds, then $\{P_Vh_k\}_{k=0}^\infty$ forms a frame for $V$.
\end{lem}

Perturbation results by themselves, however, seem to be insufficient to establish the conjecture from \cite{CHPS24}, let alone to fully understand the spanning properties of Carleson systems. In Section \ref{sec:synthesis}, we have accomplished the former goal by combining perturbation techniques with the celebrated M\"untz-Sz\'asz theorem on the completeness of monomial systems in $C[0,1]$. The latter goal is far from being achieved, although, we have made some headway in Section \ref{sec:analysis}. We intend to pursue it in our future research, perhaps, utilizing techniques from sampling and interpolation in spaces of holomorphic functions \cite{GHOR19, Levin80, NN2012} (or more general reproducing kernel Hilbert spaces \cite{FGHKR17}), more advanced perturbation techniques (such as in \cite{KM25n}), and the spectral theory of totally positive matrices \cite{Pinkus10}.  

\section{Carleson Systems: Synthesis Operator Perspective}\label{sec:synthesis}

In this section, we establish our main result by working with the synthesis operator of a Carleson system. For sequences of the form $\{T^{Nk+j_k}\varphi\}_{k =0}^\infty$, it is clear that the synthesis operator is well defined and bounded (i.e.~they are Bessel sequences); thus, to establish the conjecture, it only remains to show its surjectivity. We note that in functional analysis it is typically much easier to obtain boundedness below of the adjoint operator than prove the surjectivity directly. In this case, it turned out to be possible to prove the surjectivity by definition whereas our attempts to use the analysis operator perspective have yielded results that were not strong enough for this purpose (nevertheless, some of those results are useful and will be presented in the following section). 

It is most convenient to represent Carleson systems in $\ell^2(\N_0)$. Let $\delta_j$ denote the $j^{th}$ standard basis vector in $\ell^2(\N_0)$, and let $D \in B(\ell^2(\N_0))$ be the diagonal operator $D\delta_j = z_j\delta_j$ where $\mathcal Z = \{z_j\}_{j \in \N_0}$ is a Carleson sequence. Then, by the above, $\{D^kg\}_{k=0}^\infty$ is a frame for $\ell^2(\N_0)$ when 
\[
g = (g_j)_{j=0}^\infty = \left(m_j\sqrt{1-|z_j|^2}\right)_{j=0}^\infty.
\]
As mentioned in the introduction, by passing to a change of basis, it will suffice for us to only consider
\begin{equation}\label{eq_g}
    g = (g_j)_{j=0}^\infty = \left(\sqrt{1-|z_j|^2}\right)_{j=0}^\infty.
\end{equation}
We will refer to such frames $\{D^kg\}_{k=0}^\infty$ as \emph{canonical} Carleson frames. Clearly, each Carleson frame has a canonical representation.

Since the matrix of the operator $D$ is diagonal, there is typically no reason to restrict our investigation to merely integer powers of $D$. Thus, in general, we will consider the systems of the form $\{D^\lambda g\}_{\lambda\in \Lambda}$, where $\Lambda = \{\lambda_0,\lambda_1,...\} \subset [0,\infty)$ is a countable infinite set and $D^\lambda \delta_j = z_j^\lambda \delta_j = e^{\lambda(\log r_j+i\theta_j)}\delta_j$, $z_j = r_je^{i\theta_j}\neq 0$, $r_j> 0$, $\theta_j \in [-\pi, \pi)$, $j\in\N_0$.

With the above notation, we may identify the synthesis operator of the sequence $\{D^{\lambda_k}g\}_{k=0}^\infty$ with an infinite matrix $\Phi^\Lambda = \Phi^\Lambda_{\mathcal Z}$ with the entries $\phi_{jk}^\Lambda$ given by  
\begin{equation} \label{synop}
\phi_{jk}^\Lambda=z_j^{\lambda_k}\sqrt{1-|z_j|^2}, \ j, k \in \N_0.  \end{equation}



In this section, we will assume that the Carleson spectrum $\{z_j\}_{j = 0}^\infty$ satisfies $z_j \in (0,1)$ for each $j\in\N_0$. We also assume that $z_j > z_{j-1}$ for all $j\in\N$.


Our proof of the surjectivity of the synthesis operator has two key ingredients, the first of which is the following perturbation lemma.

\begin{lem}\label{lemma_projection}
    If $\Lambda = \{Nk + j_k\}_{k \in \N_0}$ where $N \in \N$ and $j_k \in[0, N)$, then there exists a constant $J\gg 0$ such that $\{D^{\lambda}g\}_{\lambda \in \Lambda}$ is an outer frame for $V_J := \overline{\text{span}}\{\delta_j\}_{j \geq J}$, i.e.~$\{P_{V_J}D^{\lambda}g\}_{\lambda \in \Lambda}$ forms a frame for $V_J$.
\end{lem}

\begin{proof}
    By \cite[Theorem~2.1]{CHPS24}, the sequence $\{D^{Nk}g\}_{k = 0}^\infty$ forms a frame for $\ell^2(\N_0)$. If $J \in \N$ is any number and $V_J := \overline{\text{span}}\{\delta_j\}_{j \geq J}$, we can estimate 
    \begin{equation}\label{CHPS_ineq}
    \begin{split}
        \bigg\| \sum_{k = 0}^\infty P_{V_J}(D^{Nk}g - D^{Nk+j_k}g)\bigg\|^2 &= \sum_{j = J}^\infty \bigg| \sum_{k = 0}^\infty z_j^{Nk}\sqrt{1-z_j^2} (1-z_j^{j_k}) \bigg|^2\\
        &\leq \sum_{j = J}^\infty (1-z_j^2)\bigg|\sum_{k = 0}^\infty z_j^{Nk}(1-z_j^{N})\bigg|^2 \\
        &= \sum_{j = J}^\infty (1-z_j^2) \frac{(1-z_j^{N})^2}{(1-z_j^{N})^2} \\
        &= \sum_{j = J}^\infty (1-z_j^2),
        \end{split} 
    \end{equation}
    where the inequality follows since $\mathcal Z \subset (0,1)$ and $j_k < N$. 
    
    The Carleson condition \eqref{CCond} implies that $\sum_{j = 0}^\infty (1-z_j^2) < \infty$ (see \cite[Lecture VI]{NN2012}), so for sufficiently large $J$ we have 
    \[
    \sum_{j = J}^\infty (1-z_j^2) < A,
    \]
    where $A$ is the lower frame bound for $\{D^{Nk}g\}_{k = 0}^\infty$. Thus, $\{P_{V_J}D^{\lambda}g\}_{\lambda \in \Lambda}$ forms a frame for $V_J$ by Lemma \ref{lemma_perturb}. 
\end{proof}

Note that the estimates \eqref{CHPS_ineq} were already obtained in the proof of \cite[Lemma~2.4]{CHPS24}, albeit with integer $j_k$ and a slightly different argument. However, the purpose of this calculation in \cite{CHPS24} was to argue that sequences of the form $\{D^{Nk}g\}_{k = 0}^{J-1} \cup \{D^{Nk + j_k}g\}_{k = J}^\infty$ formed a frame for sufficiently large values of $J$ (and integer $j_k$). Our reinterpretation of these estimates is crucial to begin the proof of the main theorem. To finish it off, we utilize Lemma \ref{MSlemma} below. 

We prefer to state Lemma \ref{MSlemma} in a more general form than is needed for the proof of the main result by permitting a larger class of sets $\Lambda$.


\begin{lem}\label{MSlemma}
    Let the set $\Lambda = \{\lambda_0,\lambda_1,...\}\subseteq [0,\infty)$ satisfy the generalized M\"untz-Sz\'asz condition
    \begin{equation}
        \label{MSCond}
        \sum_{k=0}^\infty \frac{\lambda_k}{\lambda_k^2+1} = \infty
    \end{equation}
    (all $\lambda_k$ are assumed to be distinct), and $\mathcal{Z}$ be an increasing Carleson sequence of positive numbers. Assume also that the function $\vartheta: [0,1)\to\CC$ given by
    \begin{equation}\label{tech_cond}
    \vartheta(z) = \sum_{k = 0}^\infty z^{\lambda_k}, \ z\in [0,1),
    \end{equation}
    is well defined and satisfies
    \begin{equation}
        \label{techcond2}
        \sup_{z\in[0,1)} \vartheta(z){(1-z)} < \infty.
    \end{equation}
     Let $\mathcal{Z}_n = \mathcal{Z}\setminus\{z_0, z_1, \ldots, z_{n-1}\}$, $n\in\N$, and assume that the synthesis matrix $\Phi^\Lambda_{\mathcal Z_J}$ is bounded and surjective for a fixed $J\in\N$. 
    Then the synthesis matrix $\Phi^\Lambda_{\mathcal Z}$ is also bounded and surjective.
\end{lem}

\begin{proof}
     Observe that 
     assumption \eqref{techcond2} implies that the matrices $\Phi^\Lambda_{\mathcal Z_n}$, $n\in\N$, have square-summable rows with a finite supremum of their norms. Moreover, Carleson's condition \eqref{CCond} implies that they also have square-summable columns with a finite supremum of their norms. Since we assumed that $\Phi^\Lambda_{\mathcal Z_J}$ is bounded (i.e.~defines a bounded operator on $\ell^2(\N_0)$), it follows that all the matrices $\Phi^\Lambda_{\mathcal Z_n}$, $n\in\N$, are also bounded.
     Thus, to complete the proof it suffices to show that 
     the matrix $\Phi^\Lambda_{\mathcal Z_{J-1}}$ is surjective. Indeed, the lemma would then follow after finitely many successive applications of this result.

Given $c \in\ell^2(\N_0)$, let $\vartheta_c: \DD\to \CC$ denote the  function 
\[
\vartheta_c(z) = \sum_{k=0}^\infty c_k z^{\lambda_k}.
\]
Clearly, $\vartheta_c$ is well defined and analytic by the assumption on $\vartheta$. 
Identifying the matrix $\Phi^\Lambda_{\mathcal{Z}_n}$ with the operator from $\ell^2(\N_0)$ to $V_n$, we get
\begin{equation*}
    \Phi^\Lambda_{\mathcal{Z}_n} c 
    = \sum_{j = n}^\infty\sum_{k=0}^\infty c_k  z_j^{\lambda_k}\sqrt{1-z_j^2}\delta_j 
    = \sum_{j = n}^\infty \vartheta_c(z_j)\sqrt{1-z_j^2}\delta_j, \ n\in\N. 
\end{equation*}

We will establish the surjectivity of the operator $\Phi^\Lambda_{\mathcal{Z}_{J-1}}:\ell^2(\N_0) \to V_{J-1}$ utilizing the surjectivity of $\Phi^\Lambda_{\mathcal{Z}_{J}}:\ell^2(\N_0) \to V_{J}$.
It suffices to show that $\delta_{J-1}$ is in the image of  
the operator $\Phi^\Lambda_{\mathcal{Z}_{J-1}}$.

     Our assumptions imply that the system
    $\{P_{V_J}D^{\lambda_k}g\}_{k=0}^\infty$ forms a frame for $V_J := \overline{\text{span}}\{\delta_j\}_{j \geq J}$; we will use its frame bounds $B \geq A > 0$. In particular, the operator $\Phi^\Lambda_{\mathcal{Z}_{J}}(\Phi^\Lambda_{\mathcal{Z}_{J}})^*$ is invertible and we have
    \[
    \left\|(\Phi^\Lambda_{\mathcal{Z}_{J}})^*\left(\Phi^\Lambda_{\mathcal{Z}_{J}}(\Phi^\Lambda_{\mathcal{Z}_{J}})^*\right)^{-1} \right\| \le \frac{\sqrt{B}}{A}.
    \]

Define $M:= \sqrt{\sum_{j = 0}^\infty (1-z_j^2)}$ and
    \begin{equation}\label{gamest}
    \gamma := 
    \sqrt{(1-z_{J-1}^2)\vartheta(z^2_{J-1})},
    \end{equation}
and choose $\varepsilon > 0$ such that 
\[
\varepsilon < \frac{1}{2}\bigg(1+\gamma M\frac{\sqrt{B}}{A}\bigg)^{-1}.
\]
By the full M\"untz-Sz\'asz theorem 
 \cite{BE96}, we can find a continuous function 
\[
\vartheta_b(z) = \sum_{k = 0}^\infty b_kz^{\lambda_k}, \ z\in [0,1],
\]
(where $b = (b_k)_{k=0}^\infty$ is finitely supported) such that 
\[
\left|\sqrt{1-z_{J-1}^2}\vartheta_b(z_{J-1}) - 1\right| < \varepsilon
\]
and $|\vartheta_b(z_j)| < \varepsilon$ for all $j \geq J$. 
Let \[c = (\Phi^\Lambda_{\mathcal{Z}_{J}})^*\left(\Phi^\Lambda_{\mathcal{Z}_{J}}(\Phi^\Lambda_{\mathcal{Z}_{J}})^*\right)^{-1}\Phi^\Lambda_{\mathcal{Z}_{J}}b\] so that
 $c = (c_k)_{k=0}^\infty \in \ell^2(\N_0)$ satisfies 
\[
\Phi^\Lambda_{\mathcal{Z}_J} c = \Phi^\Lambda_{\mathcal{Z}_{J}}b = \sum_{j=J}^\infty \vartheta_b(z_j)\sqrt{1-z_j^2}\delta_j
\]
and  
\[
\|c\| \leq \frac{\sqrt{B}}{A} \bigg(\sum_{j = J}^\infty |\vartheta_b(z_j)|^2(1-z_j^2)\bigg)^{1/2} < \frac{\sqrt{B}}{A} \varepsilon M.
\]
Then 
\begin{align*}
    \Phi^\Lambda_{\mathcal{Z}_{J-1}}(b - c) &= \sqrt{1-z_{J-1}^2}\big(\vartheta_b(z_{J-1})-\vartheta_{c}(z_{J-1})\big)\delta_{J-1},
\end{align*}
and 
\begin{align*}
    \bigg|\sqrt{1-z_{J-1}^2}\big(\vartheta_b(z_{J-1}) &-\vartheta_{c}(z_{J-1})\big) \bigg| 
    \\
    &\geq \sqrt{1-z_{J-1}^2}|\vartheta_b(z_{J-1})| - \sqrt{1-z_{J-1}^2}\bigg|\sum_{k = 0}^\infty c_kz^{\lambda_k}_{J-1}\bigg| \\
    &> 1-\varepsilon - \|c\|\sqrt{(1-z_{J-1}^2)\sum_{k = 0}^\infty z^{2\lambda_k}_{J-1}} \\
    &\geq 1 - \varepsilon - \gamma\varepsilon M \frac{\sqrt{B}}{A}
    > 0
\end{align*}
by the choice of $\varepsilon$. We conclude that $\delta_{J-1}$
is in the range of $\Phi^\Lambda_{\mathcal{Z}_{J-1}}$, and the proof is complete.
\end{proof}

Our main result follows as a combination of Lemmas \ref{lemma_projection} and \ref{MSlemma}. Indeed, from Lemma \ref{lemma_projection}, we deduce that the matrix obtained by removing the first $J$ rows of the synthesis matrix is surjective. Lemma \ref{MSlemma} then allows us to add the deleted rows of the matrix one-by-one while preserving the surjectivity. The frame property then follows by Proposition \ref{proposition_frame_equivalent}.

\begin{thm}\label{theorem_main}
    Let $\Lambda = \{Nk + j_k\}_{k \in \N_0}\subset [0,\infty)$,  where $N \in \N$ and $j_k \in[0, N)$. Then the sequence $\{D^\lambda g\}_{\lambda \in \Lambda}$ forms a frame for $\ell^2(\N_0)$.  
\end{thm}

\begin{proof}
We only need to verify that the assumptions of Lemma \ref{MSlemma} hold for $\Lambda = \{Nk + j_k\}_{k \in \N_0}$.
The generalized M\"untz-Sz\'asz condition holds for $\Lambda$ by comparison with the harmonic series. 
The condition on $\vartheta$ follows from 
\[
\sum_{k = 0}^\infty z^{\lambda_k} \le \sum_{k = 0}^\infty z^{k} = \frac1{1-z},\ z \in [0,1),
\]
so that $\sup_{z\in[0,1)} \vartheta(z){(1-z)} \le 1$ in \eqref{techcond2} and $\gamma \le 1$ in \eqref{gamest}. Finally, boundedness of the matrix $\Phi^\Lambda_{\mathcal Z_J}$ for some (and, hence, all) $J\in\N$ follows from Lemma \ref{lemma_projection}.
\end{proof}

\begin{cor}\label{lemma_general_lambda}
    Let $\Lambda \subset [0,\infty)$ be a countable subset and suppose that there exist constants $N,M \in \N$ such that $0 < |\Lambda \cap [Nk, N(k+1))| < M  $ for all $k \in \N_0$. Then $\{D^{\lambda}g\}_{\lambda \in \Lambda}$ forms a frame for $\ell^2(\N_0)$. 
\end{cor}

\begin{proof}
    The hypotheses guarantee that $\{D^{\lambda}g\}_{\lambda \in \Lambda}$ is a Bessel sequence and $\Lambda$ contains a subsequence $\lambda_0,\lambda_1,...$ such that $0 \leq \lambda_k - Nk < N-1$ for all $k \in \N_0$. It follows from Theorem \ref{theorem_main} that $\{D^{\lambda}g\}_{\lambda \in \Lambda}$ has a lower frame bound.
\end{proof}

The results of this section can be viewed as a strengthening of Carleson's original Theorem on Bases and Interpolation \cite{NN2012} for positive real Carleson sequences. By $H^2 = H^2(\DD)$ we denote the Hardy space of analytic functions of $\DD$ with square-summable power series. Given a sequence $\mathcal Z = \{z_j\}_{j \in \N_0}$ satisfying $\sum_{j} 1-|z_j| < \infty$ and $z_j \neq 0$, we may form the Blaschke product 
\[
B_{\mathcal Z}(z) = \prod_{j \in \N_0} \frac{z_j}{|z_j|} \frac{z-z_j}{1-\overline{z_j}z}.
\]
Note that any Carleson sequence $\{z_j\}_{j\in \N_0}$ automatically satisfies $\sum_{j} 1-|z_j| < \infty$. Finally, given $\Lambda \subset \N_0$, let us denote $H^2_\Lambda = \overline{\text{span}}\{z^\lambda\}_{\lambda \in \Lambda} \subset H^2$.

\begin{thm}
    Let $\mathcal Z = \{z_j\}_{j \in \N_0}$ be a sequence of positive real numbers with each $z_j \in (0,1)$. Suppose that $\Lambda \subset \N_0$ and  there exists a constant $N \in \N$ such that $|\Lambda \cap [Nk, N(k+1))| > 0$ for all $k \in \N_0$. Then the following statements are equivalent. 
    \begin{enumerate}
        \item The sequence $\mathcal Z$ is a Carleson sequence. 

        \item The set of functions 
        \[
        \bigg\{\frac{B_{\mathcal Z}}{z-z_j} \bigg\}_{j \in \N_0}
        \]
        forms a Riesz basis for $H^2 \ominus B_{\mathcal Z}H^2$.
        
        \item Given any $(a_j)_{j = 0}^\infty \in \ell^2(\N_0)$, there exists $f \in H^2_\Lambda$ such that $a_j = f(z_j)\sqrt{1-z_j^2}$ for all $j \in \N_0$. 
    \end{enumerate}
\end{thm}

\begin{rem}
    \emph{Consider a (not necessarily real) Carleson sequence 
    $\mathcal Z = \{z_i\}_{i\in\N_0}$ 
    such that $z_k^N = z_\ell^N$ for some $k \neq \ell$. The analysis operator corresponding to $\{D^{Nk}g\}_{k \in \N_0}$ has a matrix representation with $ij$-entry given by}
    \[
    \overline{z}^{Ni}_j\sqrt{1-|z_j|^2}.
    \]
    \emph{Hence, the $k^{th}$ column is a scalar multiple of the $\ell^{th}$ column, and thus the matrix cannot be bounded below. By Proposition \ref{proposition_frame_equivalent}, this means $\{D^{Nk}g\}_{k \in \N_0}$ cannot form a frame. This example shows that one must exercise caution when attempting to generalize the results in this section. However, in the case when $\Lambda\subset \N_0$, to  prove Lemma \ref{MSlemma} one only needs to find a function $\vartheta \in H^2_\Lambda$ such that $\vartheta(z_{J-1}) \neq 0$ and $\vartheta(z_j) = 0$ for all $j\geq J$. This suggests that the proof technique may be extended beyond positive real Carleson sequences, as the next proposition shows.}
\end{rem}

\begin{prop}
    Fix $J \in \N$ and let $\mathcal Z_J = \{z_j\}_{j \geq J}$ be a Carleson sequence with $z_j \in \DD$ and $|z_j| < |z_{j+1}|$ for all $j \geq J$. Assume also that $\Lambda \subset \N_0$. 
    If the matrix $\Phi^\Lambda_{\mathcal Z_J}$ is surjective, then so is the matrix $\Phi^\Lambda_{\mathcal Z}$ where $\mathcal{Z} = \{z_0,...,z_{J-1}\} \cup \mathcal{Z}_J$ for all but countably many choices of $z_0,...,z_{J-1} \in \DD$. 
\end{prop}

\begin{proof}
    An easy calculation shows that the columns of the matrix $\Phi^\Lambda_{\mathcal Z_J}$ tend to $0$, which implies that $\Phi^\Lambda_{\mathcal Z_J}$ is not injective since it is the synthesis operator for a frame. This means we may choose $0 \neq c =  (c_j)_{j =0}^\infty \in \text{Ker}(\Phi^\Lambda_{\mathcal{Z}_J})$; that is 
    \[
    0 = (\Phi^\Lambda_{\mathcal{Z}_J}c)_j = \vartheta_c(z_j) = \sum_{\lambda \in \Lambda} c_{\lambda}z_j^{\lambda}  
    \]
     for all $j \geq J$. If $\vartheta_c(z_{J-1}) \neq 0$, then 
    \[
    \Phi^\Lambda_{\mathcal{Z}_{J-1}} c = \vartheta_c(z_{J-1})\sqrt{1-|z_{J-1}|^2}\delta_{J-1}\neq 0.
    \]
    Thus $\delta_{J-1}\in \text{Im}(\Phi^\Lambda_{\mathcal{Z}_{J-1}})$, so $\Phi^\Lambda_{\mathcal{Z}_{J-1}}$ is surjective. Finally, we only need to note that $\vartheta_c \in H^2$, and hence has a countable zero set. The result now follows by induction. 
\end{proof}

To clarify the above result, we note that it is an analog of Lemma \ref{MSlemma}, which allows us to add complex values to the Carleson spectrum as long as we avoid a small pathological set.  We also
note that this pathological set of values for $z_0,...,z_{J-1}$  is a subset of a finite union of the zero sets of nonzero $H^2$ functions. Therefore, not only is the set countable as stated in the proposition, but in fact it finitely intersects each open ball $B_r(0)$ for $r < 1$. The strength of Lemma \ref{MSlemma} is in the argument that shows that this set does not intersect the interval $(0,1)$. 

\section{Carleson Systems: Analysis Operator Perspective}\label{sec:analysis}

The discrepancy in the assumptions on the set $\Lambda \subset [0,\infty)$ in Lemmas \ref{lemma_projection} and \ref{MSlemma} warrants additional investigation. In this section, we use the analysis operator $(\Phi^\Lambda_{\mathcal Z})^*$, $\mathcal Z \subset \DD$, to establish the completeness of the system $\{D^\lambda g\}_{\lambda \in \Lambda}$ for a set $\Lambda$ of sufficiently big logarithmic block density.

\begin{defn}
    For a set $\Lambda \subset [0,\infty)$, its \emph{logarithmic block density} $L(\Lambda)$ is defined by
    \begin{equation}
        L(\Lambda) = \inf_{\mu > 1} \frac{1}{\log\mu}\limsup_{t\to\infty} \sum_{\lambda\in\Lambda\cap[t, \mu t]}\frac1\lambda.
    \end{equation}
\end{defn}

Observe that $L(\Lambda) > 0$ implies the generalized M\"untz-Sz\'asz condition. For sets that satisfy the assumption of Lemma \ref{lemma_projection}, we have $L(\Lambda) = \frac1N$. 

The main theorem of this section is based on the following result, which L.~A.~Rubel established in \cite[\S 3]{R56} en route to generalizing   Carlson's Theorem on Entire Functions. 

\begin{lem}\label{lemma_rubel}
    Assume that $\Lambda \subset [0,\infty)$ and $f$ is a non-zero entire function satisfying the following three conditions:
    \begin{enumerate}
        \item $|f(z)| \le Me^{\tau|z|}$, $z\in\CC$, for some $M, \tau > 0$;
        \item $|f(iy)| \le Me^{c|y|}$, $y\in\R$, for some $M > 0$ and $c < \pi$;
        \item $f(\lambda)  = 0$ for all $\lambda\in\Lambda$.
    \end{enumerate}
    Then $L(\Lambda) \le \frac c\pi$.
\end{lem}

\begin{rem}
    In his derivations in \cite[\S 3]{R56}, Rubel was only concerned with the case of $\Lambda \subseteq \N$. This assumption, however, is unnecessary as follows from his computations and remark in \cite[\S 4]{R56}.
\end{rem}

The conclusion of Lemma \ref{lemma_rubel} can be strengthened if the function $f$ is bounded on the imaginary access as follows from \cite[Theorem~V.11]{Levin80}.

\begin{lem}
    \label{lemma_levin}
    Suppose that the assumptions of Lemma \ref{lemma_rubel} hold with $c=0$. Then
    \[
    \sum_{\lambda\in \Lambda\setminus\{0\}}\frac1\lambda < \infty.
    \]
\end{lem}

Lemma \ref{lemma_rubel} is all we need to establish the following completeness result.

\begin{thm}\label{theorem_completeness}
    Assume that $\mathcal Z \subset\DD_c$ is a Carleson sequence, where
    \[
    \DD_c = \{z = re^{i\theta}\in \DD\setminus\{0\}: \theta\in[-c, c]\}, \ c\in [0,\pi).
    \]
    Assume also that $\Lambda 
    \subset [0,\infty)$ satisfies $L(\Lambda) > \frac c\pi$.
    Then the system $\{D^\lambda g\}_{\lambda \in \Lambda}$ is complete in $\ell^2(\N_0)$.
\end{thm}

\begin{proof}
We only consider countable sets $\Lambda$ (there is no loss of generality as we may always pass to a countable subset).    
Assume towards contradiction that the system $\{D^\lambda g\}_{\lambda \in \Lambda}$ is not complete. Then there exists a nonzero $b = (b_j)_{j=0}^\infty \in \ell^2(\N_0)$ such that $\langle b, D^\lambda g\rangle = 0$ for each $\lambda\in \Lambda$, i.e.~the analysis operator of $\{D^\lambda g\}_{\lambda \in \Lambda}$ satisfies $(\Phi^\Lambda_{\mathcal Z})^*b = 0$. Consider the function $\psi: \CC\to \CC$ given by
    \[
    \begin{split}
        \psi(\omega) = \langle b, D^\omega g\rangle & = \sum_{j\in\N_0} b_j\bar z_j^{\omega}\sqrt{1-|z_j|^2} \\
        &=  \sum_{j\in\N_0} \beta_j e^{\omega(\log r_j -i\theta_j)}, \ \beta_j = b_j\sqrt{1-|z_j|^2}.
    \end{split}
    \]
    Since $\beta = (\beta_j)_{j = 0}^\infty\in \ell^1(\N_0)$ by the Cauchy-Schwarz inequality, $\psi$ is an entire function of exponential type. We also have $\psi\not\equiv 0$ since the system $\{D^\lambda g\}_{\lambda \in \N_0}$ is a frame (we must have $\psi(n) \neq 0$ for at least one $n\in\N_0$). Thus, $\psi$ is a nonzero entire function that satisfies Assumption (1) of Lemma \ref{lemma_rubel}. Moreover, Assumption (2) of the lemma follows from $\mathcal Z \subset \mathbb D_c$, and Assumption (3) follows by construction since $\psi(\lambda) = \langle b, D^\lambda g\rangle = 0$, $\lambda\in\Lambda$. Thus, we arrived at a contradiction with the density assumption  $L(\Lambda) > \frac c\pi$  and the proof is complete.
\end{proof}

It is immediate from the above proof that Lemma \ref{lemma_levin} allows us to strengthen the result in the case when $\mathcal{Z}\subset (0,1)$.

\begin{thm}
    Assume that $\mathcal Z \subset(0,1)$ is a Carleson sequence and $\Lambda \subset [0,\infty)$ satisfies the M\"untz-Sz\'asz condition \[\sum_{\lambda\in \Lambda\setminus\{0\}}\frac1\lambda = \infty.\] 
    Then the system $\{D^\lambda g\}_{\lambda \in \Lambda}$ is complete in $\ell^2(\N_0)$.
\end{thm}

The results of this section underscore the relevance of the sampling and interpolation theory of entire functions in the context of Carleson systems. We intend to pursue this avenue of research in our future investigation.

\section{Frame Bounds and Continuous Carleson Frames}

In this section, we discuss the bounds of discrete and continuous Carleson frames.

Let us recall the definition of a continuous frame. 

\begin{defn}
    Given a measure space $(X,\mu)$, a family $\{f_t\}_{t\in X}$ of vectors in $\ell^2(\N_0)$ is said to be a continuous frame if the function $f \mapsto \ang{f,f_t}$ is measurable for each $t \in X$, and there exist constants $B\geq A>0$ such that 
    \[
    A\|f\|^2 \leq \int_X |\ang{f,f_t}|^2 d\mu(t) \leq B\|f\|^2
    \]
    for all $f \in \ell^2(\N_0)$. 
\end{defn}
\noindent For our purposes, we will only be interested in $X = [0,\infty)$ endowed with the Lebesgue measure. 

\begin{thm}[\protect{\cite[Theorem~5.1]{CHP20}}]\label{theorem_estimates}
    Let $\{z_j\}_{j \in \N_0}$ be a Carleson sequence with $z_j \in \DD$ and set 
    \[
    \delta : = \inf_{k} \prod_{j\neq k} \bigg|\frac{z_k - z_j}{1 - z_j\overline{z_k}} \bigg|.
    \]
    If $D \in B(\ell^2(\N_0))$ is the diagonal operator $D\delta_j = z_j\delta_j$, and $g$ is as in \eqref{eq_g}, then $\{D^kg\}_{k = 0}^\infty$ forms a frame for $\ell^2(\N_0)$ with frame bounds $\Delta \geq \Delta^{-1} > 0$ where 
    \begin{equation}
    \Delta := \frac{2(1 - 2\log{\delta)}}{\delta^4}.
    \end{equation}
\end{thm}

Using the above estimates for the frame bounds, we  show that for positive real Carleson sequences, the family $\{D^tg\}_{t \in [0,\infty)}$ forms a continuous frame and compute its frame bounds. 

\begin{thm}
    Let $\{z_j\}_{j \in \N_0}$ be an increasing positive real Carleson sequence, and let $D$ and $g$ be as before. Then, the family $\{D^tg\}_{t \in [0,\infty)}$ forms a continuous frame with 
    \[
    \frac{\Delta^{-1}(z_0^2-1)}{2\log{z_0}}\|f\|^2 \leq \int_0^\infty |\ang{f,D^tg}|^2 dt \leq \Delta\|f\|^2.
    \]
\end{thm}

\begin{proof}
    Let $\Delta$ be as in Theorem \ref{theorem_estimates}. Using Riemann sums, we can write 
    \begin{align*}
        \int_0^\infty |\ang{f,D^tg}|^2 \;dt &= \lim_{N\to \infty } \sum_{n = 0}^\infty \frac{1}{N}|\ang{f,D^{n/N}g}|^2 \\
        &= \lim_{N\to \infty} \sum_{n=0}^\infty \sum_{k = 0}^{N-1} \frac{1}{N}|\ang{f,D^{n + k/N}g}|^2 \\
        &= \lim_{N\to \infty} \sum_{k = 0}^{N-1} \sum_{n=0}^\infty  \frac{1}{N}|\ang{D^{k/N}f,D^{n}g}|^2 .\\ 
    \end{align*}
    By Theorem \ref{theorem_estimates}, we have that 
    \[
    \sum_{n=0}^\infty  |\ang{D^{k/N}f,D^{n}g}|^2 \leq \Delta \|D^{k/N}f\|^2 \leq \Delta \|f\|^2.
    \]
    Therefore, it follows that $\Delta$ is an upper frame bound for $\{D^tg\}_{t \in [0,\infty)}$. Similarly, we can compute 
    \begin{align*}
     \sum_{n=0}^\infty  |\ang{D^{k/N}f,D^{n}g}|^2 &\geq \sum_{n=0}^\infty \Delta^{-1} \|D^{k/N}f\|^2,  
    \end{align*}
    so that
    \begin{align*}
        \lim_{N\to \infty} \sum_{k = 0}^{N-1} \sum_{n=0}^\infty  \frac{1}{N}|\ang{D^{k/N}f,D^{n}g}|^2 &\geq \lim_{N\to \infty}\sum_{k=0}^{N-1}\frac{1}{N} \Delta^{-1} \|D^{k/N}f\|^2 \\
        &\geq \lim_{N\to \infty}\Delta^{-1}\|f\|^2\sum_{k=0}^{N-1}\frac{1}{N}  z_0^{2k/N} \\
        &= \Delta^{-1}\|f\|^2 \lim_{N\to \infty} \frac{1-z_0^2}{N(1-z_0^{2/N})} \\
        &= \frac{\Delta^{-1}(z_0^2-1)}{2\log{z_0}}\|f\|^2.
    \end{align*}
    This completes the proof.
\end{proof}

\medskip
\noindent {\bf{Acknowledgments}.} Both authors of the paper were supported in part by the NSF grant DMS-2208031. We thank A.~Aldroubi, C.~Heil,  A.~Fletcher, and J.~Mashreghi for helpful discussions. Special thanks are reserved for O.~Christensen and M.~Hasannasab for introducing their conjecture and sharing their ideas on related problems. We are also grateful to the anonymous referee for helping us improve the manuscript. 

\bibliographystyle{abbrv}
\bibliography{refs1}

\end{document}